\documentclass[11pt]{article}
\usepackage[margin=1.1in]{geometry}
\usepackage{amsmath,amssymb,amsthm}
\usepackage{booktabs}
\usepackage{array}
\usepackage[hidelinks]{hyperref}
\usepackage{orcidlink}

\newtheorem{theorem}{Theorem}
\newtheorem{corollary}[theorem]{Corollary}
\newtheorem*{corollary*}{Corollary}
\theoremstyle{remark}

\newcommand{\AC}{\sim_{\mathrm{AC}}}
\newcommand{\dGS}{d_{\mathrm{GS}}}
\newcommand{\MS}{\mathrm{MS}}
\newcommand{\AK}{\mathrm{AK}}
\newcommand{\PP}[1]{P_{#1}}
\newcommand{\code}[1]{\texttt{#1}}

\title{Machine-checkable equivalence certificates\\
at the length-14 Andrews--Curtis frontier}
\author{Josep Carreras\,\orcidlink{0000-0003-4063-3243}\thanks{Independent
Researcher. Email: \texttt{joe.carr.data@gmail.com}. ORCID:
\texttt{0000-0003-4063-3243}. This work was carried out in collaboration with
AI systems; in line with publisher and preprint-server policies, AI systems are not
listed as authors, and their contributions are described in the disclosure statement
(\S\ref{sec:repro}).}}
\date{July 22, 2026}

\begin{document}
\maketitle

\begin{abstract}
In rank 2, unconditional verification of the Andrews--Curtis conjecture stands
at total relator length 12: every balanced presentation of the trivial group
on two generators with total length at most 12 is AC-trivializable
(Miasnikov--Myasnikov); at length 13, every such presentation is
AC-trivializable or AC-equivalent to the Akbulut--Kirby presentation
$\AK(3)$ --- itself of length 13 and open (Havas--Ramsay). At length 14, the
reinforcement-learning campaign of Shehper et al.\ reduced the Miller--Schupp
family to six hard presentations: four stated to be AC-equivalent to $\AK(3)$
--- with no move sequences published --- and two left unresolved. We prove
four explicit AC-equivalences among these six, as machine-checkable
elementary-move certificates replayable in under one second by a small
dependency-free verifier: (1)~the classical candidate
$\PP{1} = \langle x,y \mid x^{-1}y^2x=y^3,\ y^{-1}x^2y=x^3\rangle$ (Johnson 1980; an
AC-candidate since Burns--Macedo\'nska 1993, credited to D.~Solitar), which
equals $\MS(2, x^{-2}y^{-1}x^2y)$ up to rotation, is AC-equivalent to the
unresolved presentation $\MS(2, yx^2y^{-1}x^{-2})$, in 36 packaged moves (58
classical); (2)~the other unresolved presentation $\MS(2, yx^2yx^{-2})$ is
AC-equivalent to $\MS(2, x^{-2}y^{-1}x^2y^{-1})$, in 85 packaged moves (155
classical); (3)~$\MS(3, yx^2y)$ is AC-equivalent to $\AK(3)$, in 66 packaged
moves (101 classical); (4)~$\MS(3, y^{-1}x^2y^{-1})$ is AC-equivalent to
$\AK(3)$, in 13 packaged moves (22 classical). Theorems 3--4 are, to our
knowledge, the first public explicit certificates for any of the four
$\AK(3)$-equivalences stated without proof by Shehper et al.; they render the
$\MS(3)$ branch of the length-14 collapse unconditional. The pairs of
Theorems 1--2 are exchanged by the automorphism $\sigma\colon x\mapsto x$,
$y\mapsto y^{-1}$, so those theorems state that $\sigma$ is AC-realized on
the two remaining open classes --- the analogue, at the open frontier, of
Panteleev--Ushakov's realization of $\operatorname{Aut}(F_2)$ on $\AK(n)$.
The certified connected blocks of the six-case benchmark are now
$\{\MS(3,yx^2y),\, \MS(3,y^{-1}x^2y^{-1}),\, \AK(3)\}$,
$\{\PP{1},\, \MS(2,yx^2y^{-1}x^{-2})\}$, and
$\{\MS(2,yx^2yx^{-2}),\, \MS(2,x^{-2}y^{-1}x^2y^{-1})\}$; the full collapse
to $\AK(3)$ now depends on exactly the two remaining uncertified $\MS(2)$
claims of Shehper et al. We complement the certificates with a
computer-assisted exhaustive minimax analysis of the canonical
substitution-move (GS) graph: the GS bottleneck distance from $\MS(3, yx^2y)$
to $\AK(3)$ is exactly 19, while any GS path joining
$\MS(2, x^{-2}y^{-1}x^2y)$ or $\MS(2, x^{-2}y^{-1}x^2y^{-1})$ to $\AK(3)$
--- or to each other --- must reach total relator length at least 27
(exhaustions of 6{,}212{,}968 and 13{,}504{,}944 canonical states). We
further analyze the public classification table of the ``Two-Hump'' campaign
(Fagan et al.), derive a $\sigma$-merge program of 214 explicit candidate
class mergers, and commit a membership audit of their AC-19 solved set
confirming all six benchmark presentations and $\AK(3)$ absent. All
certificates, both search engines, the verifier, and one-command reproduction
scripts are archived at \url{https://github.com/joe-carr-data/ac-certificates} (tag \code{v1.0}) and
archived at \url{https://doi.org/10.5281/zenodo.21499081} (DOI \code{10.5281/zenodo.21499081}).
\end{abstract}

\section{Introduction}\label{sec:intro}

A balanced presentation $\langle x_1,\dots,x_n \mid r_1,\dots,r_n\rangle$ of
the trivial group is \emph{AC-trivializable} if it can be transformed into
$\langle x_1,\dots,x_n \mid x_1,\dots,x_n\rangle$ by a sequence of elementary
moves: replacing some $r_i$ by $r_i^{-1}$, by $r_i r_j^{\pm1}$ ($j\neq i$),
or by $w r_i w^{-1}$ for a word $w$. The Andrews--Curtis conjecture
\cite{AC65} asserts every balanced presentation of the trivial group is
AC-trivializable; it has stood open since 1965, with connections to the
Zeeman conjecture and to 4-manifold topology via the Akbulut--Kirby
presentations $\AK(n) = \langle x,y \mid x^n = y^{n+1},\ xyx = yxy\rangle$
\cite{AK85}.

\paragraph{The rank-2 frontier.}
Miasnikov--Myasnikov \cite{MM03} proved that every balanced two-generator
presentation of the trivial group with total relator length $\le 12$ is
AC-trivializable. Havas--Ramsay \cite{HR03} showed that at length 13 every
such presentation is AC-trivializable or AC-equivalent to $\AK(3)$; since
$\AK(3)$ itself has total length 13 and remains open, $\AK(3)$ is the unique
minimal potential counterexample, and unconditional verification stops at 12.
The length-14 stratum has never been exhaustively classified. Within the
Miller--Schupp family \cite{MS99}
$\MS(n, w) = \langle x,y \mid x^{-1}y^{n}xy^{-(n+1)},\ x^{-1}w\rangle$,
Shehper et al.\ \cite[\S3.3]{Shehper24} report six length-14 presentations
that resisted their greedy search:
\begin{itemize}
\item $\MS(2, x^{-2}y^{-1}x^2y)$ and $\MS(2, x^{-2}y^{-1}x^2y^{-1})$,
\item $\MS(3, yx^2y)$ and $\MS(3, y^{-1}x^2y^{-1})$ \quad (the equal-sign
      $\MS(3)$ cases),
\item $\MS(2, yx^2yx^{-2})$ and $\MS(2, yx^2y^{-1}x^{-2})$.
\end{itemize}
They state that the first four are AC-equivalent to $\AK(3)$; the last two
are left unresolved (``the holdouts''). Crucially, \emph{no move sequences
for the four stated $\AK(3)$-equivalences are public}: the paper prints two
unrelated paths, and an audit of the released repository, its history, the
associated Weights \& Biases artifacts, the successor repository (ACSolverX),
and recorded talks found no ledgers (audit scope and date:
\S\ref{sec:repro}; an unindexed or private artifact cannot be excluded). The
determination that the six failures are exactly the presentations listed ---
in particular that the repeated $\MS(3)$ signs are correlated --- follows
from the paper's count together with its released benchmark data, in which
the two cross-sign $\MS(3)$ cases are recorded as greedy-solved.
\textbf{Theorems \ref{thm:ms3a}--\ref{thm:ms3b} below supply explicit public
certificates for the two $\MS(3)$ claims; the two $\MS(2)$ claims remain
uncertified.}

Outside the Miller--Schupp family (up to the identification below), the
classical presentation
\[
\PP{1} \;=\; \langle x, y \mid x^{-1}y^2x = y^3,\ y^{-1}x^2y = x^3\rangle
\]
appears in Johnson's 1980 lecture notes \cite[p.~41]{Johnson80} and was posed
as an Andrews--Curtis candidate --- credited to D.~Solitar --- as ``Possible
Counterexample~1'' of the Burns--Macedo\'nska survey \cite{BM93}. As relator
words, $\PP{1} = (x^{-1}y^2xy^{-3},\ y^{-1}x^2yx^{-3})$, and its second
relator is a cyclic rotation of $x^{-3}y^{-1}x^2y$; hence $\PP{1}$ coincides,
up to rotation, with $\MS(2, x^{-2}y^{-1}x^2y)$ --- the first of Shehper et
al.'s stated $\AK(3)$ cases. The ``Two-Hump'' campaign \cite{Fagan26}
enumerated all 213{,}946 balanced presentations of total length $\le 19$ and
released AC-19, a solved set of over 140{,}000 presentations; our committed
audit (\S\ref{sec:repro}) confirms under their own canonicalizer that all six
benchmark presentations and $\AK(3)$ are absent from AC-19 and from its
extended release. Thus, within the named benchmarks and our audit scope, the
length-14 presentations with no public resolution were, before this work, the
six benchmark cases; Theorems \ref{thm:ms3a}--\ref{thm:ms3b} resolve the two
$\MS(3)$ cases into $\AK(3)$'s class.

\paragraph{Contributions.}
\begin{enumerate}
\item Four explicit AC-equivalence certificates (Theorems
      \ref{thm:sigma1}--\ref{thm:ms3b}): the two $\sigma$-realization
      theorems on the open $\MS(2)$ classes, and the first public
      certificates connecting both $\MS(3)$ cases to $\AK(3)$ --- making the
      $\MS(3)$ branch of the length-14 collapse unconditional
      (Corollary~\ref{cor:collapse}'s dependency shrinks to the two $\MS(2)$
      claims).
\item A computer-assisted exhaustive minimax analysis of the canonical
      GS-substitution graph: $\dGS(\MS(3,yx^2y), \AK(3)) = 19$ exactly, while
      every GS path joining $\MS(2,x^{-2}y^{-1}x^2y)$ or
      $\MS(2,x^{-2}y^{-1}x^2y^{-1})$ to $\AK(3)$, or to each other, must
      reach total length $\ge 27$ --- a quantified asymmetry between the
      ``easy'' and ``hard'' halves of the benchmark (\S\ref{sec:minimax}).
\item An analysis of the public Two-Hump classification table (not reproduced
      by any of three tested symmetry quotients; fingerprints consistent with
      unpublished computed connections), a $\sigma$-merge program of 214
      explicit candidate class mergers, and a committed AC-19 membership
      audit (\S\ref{sec:twohump}, \S\ref{sec:repro}).
\item Verified negative results and computational practice notes, including
      a Prover9 pitfall (\code{auto\_denials}) that makes disjunctive-goal
      reachability encodings report SEARCH FAILED despite finding proofs of
      the reachable disjuncts (\S\ref{sec:negative}).
\end{enumerate}
Every positive AC-equivalence and AC-trivialization claim in this paper ships
as a machine-replayable JSON move-ledger checked by an independent verifier;
the analysis claims of \S\ref{sec:minimax}--\S\ref{sec:twohump} regenerate
from committed scripts with hard validation gates; \S\ref{sec:repro} gives
one-command reproduction.

\section{Moves, conventions, and certificates}\label{sec:conventions}

\paragraph{Words and presentations.}
Relators are freely reduced words over $\{x, y, x^{-1}, y^{-1}\}$. In the
machine-readable artifacts they are encoded as integer sequences ($1 = x$,
$2 = y$, negatives for inverses); throughout the text we write them as words.
A state is an ordered pair $(r_1, r_2)$.

\paragraph{Elementary moves.}
Our ledgers use: (i)~inversion $r_i \leftarrow r_i^{-1}$;
(ii)~multiplication $r_i \leftarrow r_i r_j^{\pm1}$ ($i \neq j$);
(iii)~conjugation by a single generator $r_i \leftarrow g r_i g^{-1}$,
$g \in \{x, y, x^{-1}, y^{-1}\}$; and (iv)~cyclic rotation of $r_i$ by $k$
positions, packaged as one move. Left rotation of a cyclically reduced
relator $r = uv$ ($u$ the length-$k$ prefix) equals conjugation by $u^{-1}$:
$u^{-1}(uv)u = vu$. \emph{Both accountings are reported everywhere}:
``packaged'' counts each move above as 1; ``classical'' counts a rotation of
a length-$L$ relator by $k$ as $\min(k \bmod L,\ L - k \bmod L)$
single-generator conjugations --- the shorter direction --- evaluated
against the relator's length \emph{at move time}, and every other move as 1.
(An earlier draft computed rotation costs against initial relator lengths;
adversarial review caught this, the verifier now tracks state lengths, and a
regression test pins the corrected counts.)

\paragraph{Orbit convention.}
The symmetry group $\Sigma$ generated by (a)~swapping $r_1 \leftrightarrow
r_2$, (b)~inverting either relator, and (c)~cyclically rotating either
relator, consists entirely of AC-realizable operations: swap is realized by
a 6-move elementary identity, inversion is move~(i), rotation is move~(iv).
An \emph{equivalence certificate} for $P \AC Q$ is a ledger transforming $P$
into some member of $\Sigma\cdot Q$; the verifier checks membership in the
full orbit explicitly. All equivalences claimed below are therefore genuine
AC-equivalences of the named presentations.

\paragraph{Certificates and verifier.}
A certificate is a JSON file with fields \code{initial}, \code{target},
\code{claim}, and \code{moves}. The verifier \code{ac\_verify.py} (small,
standard library only) re-derives every move in exact integer arithmetic,
trusts nothing about how a ledger was produced, validates the shape and
alphabet of every word field strictly, and in equivalence mode checks that
the terminal state equals the recorded \code{reached} state \emph{and} that
it lies in $\Sigma\cdot\code{target}$. The verifier was written and
self-tested before any search (``certificate-first'' discipline); its
self-test covers trivialization and equivalence modes, corrupted-ledger and
malformed-field rejection, cycle-move accounting, and a classical-count
regression.

\section{Four equivalence theorems}\label{sec:theorems}

Throughout, $\MS(2, w) = \langle x,y \mid x^{-1}y^2xy^{-3},\ x^{-1}w\rangle$
and $\MS(3, w) = \langle x,y \mid x^{-1}y^3xy^{-4},\ x^{-1}w\rangle$. The
six benchmark presentations (each of total length 14):

\begin{center}
\small
\begin{tabular}{clll}
\toprule
\# & Presentation & Status in \cite{Shehper24} & Status here \\
\midrule
$\PP{1}$ & $\MS(2, x^{-2}y^{-1}x^2y)$\,$^{\dagger}$ &
  claimed $\AK(3)$-equiv.\ (no proof) & $\AC \PP{6}$ (Thm~\ref{thm:sigma1}) \\
$\PP{2}$ & $\MS(2, x^{-2}y^{-1}x^2y^{-1})$ &
  claimed $\AK(3)$-equiv.\ (no proof) & $\AC \PP{5}$ (Thm~\ref{thm:sigma2}) \\
$\PP{3}$ & $\MS(3, yx^2y)$ &
  claimed $\AK(3)$-equiv.\ (no proof) & $\AC \AK(3)$ (Thm~\ref{thm:ms3a}) \\
$\PP{4}$ & $\MS(3, y^{-1}x^2y^{-1})$ &
  claimed $\AK(3)$-equiv.\ (no proof) & $\AC \AK(3)$ (Thm~\ref{thm:ms3b}) \\
$\PP{5}$ & $\MS(2, yx^2yx^{-2})$ & unresolved holdout &
  $\AC \PP{2}$ (Thm~\ref{thm:sigma2}) \\
$\PP{6}$ & $\MS(2, yx^2y^{-1}x^{-2})$ & unresolved holdout &
  $\AC \PP{1}$ (Thm~\ref{thm:sigma1}) \\
\bottomrule
\end{tabular}

\smallskip
{\footnotesize $^{\dagger}$\,equals the classical presentation
$\langle x,y \mid x^{-1}y^2x=y^3,\ y^{-1}x^2y=x^3\rangle$ up to cyclic
rotation of its second relator.}
\end{center}

\begin{theorem}\label{thm:sigma1}
The presentations $\PP{1} = \langle x,y \mid x^{-1}y^2x=y^3,\
y^{-1}x^2y=x^3\rangle$ and $\PP{6} = \MS(2, yx^2y^{-1}x^{-2})$ are
AC-equivalent.
\end{theorem}
\begin{proof}
Certificate \code{cert\_sym\_equiv\_orphan1.json}: 36 packaged moves (58
classical), elementary peak total length 27 (single relator 17). Replay
$< 1$~s.
\end{proof}

\begin{theorem}\label{thm:sigma2}
$\PP{5} = \MS(2, yx^2yx^{-2})$ is AC-equivalent to
$\PP{2} = \MS(2, x^{-2}y^{-1}x^2y^{-1})$.
\end{theorem}
\begin{proof}
Certificate \code{cert\_idx0\_equiv\_mirror.json}: 85 packaged moves (155
classical), elementary peak 33 (single relator 21). Replay $< 1$~s.
\end{proof}

\begin{theorem}\label{thm:ms3a}
$\PP{3} = \MS(3, yx^2y)$ is AC-equivalent to $\AK(3)$.
\end{theorem}
\begin{proof}
Certificate \code{bottleneck\_ak\_f3\_cert.json}: 66 packaged moves (101
classical), transforming $\AK(3)$ into a $\Sigma$-orbit member of
$\MS(3, yx^2y)$; elementary peak 24 (single relator 15). A second,
independently found certificate (\code{cert\_ms3\_yx2y\_equiv\_ak3.json}, 71
moves, elementary peak 22) proves the same statement. Replay $< 1$~s each.
\end{proof}

\begin{theorem}\label{thm:ms3b}
$\PP{4} = \MS(3, y^{-1}x^2y^{-1})$ is AC-equivalent to $\AK(3)$.
\end{theorem}
\begin{proof}
\sloppy
Certificate \code{cert\_ms3\_yinvx2yinv\_equiv\_ak3.json}: 13 packaged
moves (22 classical), elementary peak 23 (single relator 14). Replay
$< 1$~s.
\end{proof}

Verification data from a clean re-run (full hashes, commands, and expected
outputs: \code{MANIFEST.json}):

\begin{center}
\footnotesize
\setlength{\tabcolsep}{4.5pt}
\begin{tabular}{lccc}
\toprule
Certificate & SHA-256 (prefix) & Moves & Peak total / single \\
\midrule
\code{cert\_sym\_equiv\_orphan1.json} & \code{b8ead47e} & 36 (58) & 27 / 17 \\
\code{cert\_idx0\_equiv\_mirror.json} & \code{6f55a391} & 85 (155) & 33 / 21 \\
\code{bottleneck\_ak\_f3\_cert.json} & \code{8f16b97a} & 66 (101) & 24 / 15 \\
\code{cert\_ms3\_yx2y\_equiv\_ak3.json} & \code{9cfeff76} & 71 (106) & 22 / 15 \\
\code{cert\_ms3\_yinvx2yinv\_equiv\_ak3.json} & \code{20d9fd3d} & 13 (22) & 23 / 14 \\
\bottomrule
\end{tabular}
\end{center}

\begin{corollary*}[$\sigma$-realization at the frontier]
Let $\sigma \in \operatorname{Aut}(F_2)$ be $x\mapsto x$, $y\mapsto y^{-1}$.
The pairs $(\PP{1},\PP{6})$, $(\PP{5},\PP{2})$, and $(\PP{3},\PP{4})$ are
each exchanged by $\sigma$ up to $\Sigma$. Theorems
\ref{thm:sigma1}--\ref{thm:sigma2} therefore say that $\sigma$ is AC-realized
on the two open $\MS(2)$ classes, with explicit elementary-move witnesses ---
the frontier analogue of Panteleev--Ushakov's theorem \cite[\S3]{PU16} that
every automorphism of $F_2$ is AC-realized on $\AK(n)$; and Theorems
\ref{thm:ms3a}--\ref{thm:ms3b} show the third $\sigma$-pair lies in
$\AK(3)$'s class, where \cite{PU16} realizes all of
$\operatorname{Aut}(F_2)$.
\end{corollary*}

Panteleev--Ushakov's \S6.3 experiments found automorphic equivalences among
\emph{random} Miller--Schupp presentations (their Table~3); none of their
listed pairs involves the presentations above (checked against the full
text). To our knowledge no published work states any of Theorems
\ref{thm:sigma1}--\ref{thm:ms3b} with an explicit certificate.

\paragraph{Discovery provenance.}
All certificates were found by a three-stage pipeline: (a)~search on the
canonical substitution-move (``GS'') graph of \cite{Fagan26} --- greedy for
Theorems \ref{thm:sigma1}--\ref{thm:sigma2} and the 71-move variant of
Theorem~\ref{thm:ms3a} (which fell in seconds), and an exhaustive
bidirectional minimax (bottleneck-Dijkstra) search for the 66-move variant
(\S\ref{sec:minimax}); (b)~expansion of the resulting canonical-state path
into elementary moves by orbit-bridging breadth-first search (each
consecutive canonical pair is bridged by explicit elementary moves to an
explicit orbit member); (c)~independent replay by the verifier. The same
pipeline cross-validates two published Two-Hump trivializations at the
elementary level (\S\ref{sec:negative}).

\section{Certified blocks, and what remains conditional}\label{sec:blocks}

The certified connected blocks of the six-case benchmark are now (each block
is proven connected by the named certificates; we do not claim these are
maximal AC-components --- indeed Corollary~\ref{cor:collapse} concerns
evidence that they join):
\begin{align*}
C_0 &= \{\, \PP{3},\ \PP{4},\ \AK(3) \,\}
      && \text{(Theorems \ref{thm:ms3a}--\ref{thm:ms3b} --- unconditional)}\\
C_1 &= \{\, \PP{1},\ \PP{6} \,\}
      && \text{(Theorem \ref{thm:sigma1})}\\
C_2 &= \{\, \PP{5},\ \PP{2} \,\}
      && \text{(Theorem \ref{thm:sigma2})}
\end{align*}

\begin{corollary}[conditional on the two remaining claims of
{\cite[\S3.3]{Shehper24}}]\label{cor:collapse}
Assume, as stated but not certified in \cite[\S3.3]{Shehper24}, that
$\PP{1} = \MS(2, x^{-2}y^{-1}x^2y)$ and $\PP{2} = \MS(2,
x^{-2}y^{-1}x^2y^{-1})$ are AC-equivalent to $\AK(3)$. Then $C_1$ and $C_2$
join $C_0$, and all six hard length-14 Miller--Schupp presentations ---
including both holdouts and the classical candidate $\PP{1}$ --- lie in the
AC-class of $\AK(3)$. Within this six-case benchmark, the remaining question
is then exactly $\AK(3)$.
\end{corollary}

The dependency is now exactly two uncertified claims (reduced from four by
Theorems \ref{thm:ms3a}--\ref{thm:ms3b}), and joining the three blocks
requires two further edges: one connecting each $\MS(2)$ block to $C_0$ (or
one to $C_0$ and the two $\MS(2)$ blocks to each other). (The length-14
stratum as a whole has never been exhaustively classified, so this corollary
concerns the named benchmark.) Producing either explicit
$\MS(2) \to \AK(3)$ certificate would make the corresponding branch
unconditional; \S\ref{sec:minimax} quantifies the computational barrier our
searches have established for such certificates.

\section{Exhaustive minimax analysis:
$\dGS(\MS(3),\AK(3)) = 19$ versus $\dGS(\MS(2),\cdot) \ge 27$}
\label{sec:minimax}

Both search engines operate on the canonical GS-substitution graph of
\cite{Fagan26}: vertices are canonical states (pairs of cyclically reduced
relators, canonicalized up to $\Sigma$), edges are their substitution moves
(cyclically aligned products with a cancelling junction), and the
\emph{energy} of a vertex is its total relator length. For a path $\gamma$,
define its \emph{bottleneck} $B(\gamma) = \max_{P\in\gamma} E(P)$; for
states $S, T$ define $\dGS(S, T) = \min_{\gamma\colon S\leadsto T}
B(\gamma)$. Our bidirectional multi-source bottleneck-Dijkstra
implementation settles states in nondecreasing bottleneck order (with
collision candidates scored on both sides' labels and optimal termination),
so exhaustion of all states below a bound is a complete search of that
energy sublevel. Two implementations --- Python and an $11\times$ faster
C\texttt{++} port --- agree bit-for-bit on regression gates (identical
collision barrier on a known-equivalent pair; identical component counts at
exclusive cap 25).

\begin{theorem}[computer-assisted; engine-relative]\label{thm:minimax}
In the canonical GS-substitution graph, with total cyclically reduced
relator length as vertex energy:
\[
\dGS(\MS(3, yx^2y),\ \AK(3)) = 19,
\]
while
\[
\dGS(\PP{1}, \AK(3)) \ge 27, \qquad
\dGS(\PP{2}, \AK(3)) \ge 27, \qquad
\dGS(\PP{1}, \PP{2}) \ge 27,
\]
where $\PP{1} = \MS(2, x^{-2}y^{-1}x^2y)$ and $\PP{2} = \MS(2,
x^{-2}y^{-1}x^2y^{-1})$ are the searched representatives.
\end{theorem}

The bounds are stated for these two presentations only: their block-mates
$\PP{6}$ and $\PP{5}$ are AC-equivalent to them by Theorems
\ref{thm:sigma1}--\ref{thm:sigma2}, but elementary AC-equivalence does not
transfer $\dGS$ bounds, so no bound is asserted for the other
representatives. Thus the two searched $\MS(2)$ representatives have GS
bottleneck at least eight total-length units above the exactly determined
$\MS(3)$ bottleneck.

\begin{proof}[Computer-assisted proof]
Upper bound and exactness for $\MS(3)$: the minimax search from
$\{\AK(3),\ \MS(3,yx^2y)\}$ collided at barrier 19 after settling all states
of lower barrier; the barrier-19 GS path expands to the verified 66-move
certificate of Theorem~\ref{thm:ms3a}. Lower bounds: the exclusive-cap-26
search from the three sources $\{\AK(3),\ \PP{1},\ \PP{2}\}$ exhausted
6{,}212{,}968 canonical states without any cross-source intersection,
excluding every GS path of bottleneck $\le 25$; the exclusive-cap-27 search
exhausted 13{,}504{,}944 states, excluding bottleneck $\le 26$. A cap-28 run
reached its 40-million-state budget while processing barrier 27 and is
inconclusive for bottleneck 27.
\end{proof}

\paragraph{Scope restrictions (essential).}
These are statements about the GS-substitution graph: they are not
AC-distance lower bounds, not elementary-path peak bounds, not move-count
bounds, and not evidence of AC-inequivalence. The GS generator emits only
products with a cancelling junction, so legal elementary multiplications
without such a junction are outside this graph; the corridor joining
$\PP{1}$ and $\PP{2}$ to $\AK(3)$ --- which exists if the two remaining
claims of \cite{Shehper24} are correct --- may open at GS bottleneck 27, at
a higher level, or only through non-GS elementary moves. Likewise, 19 is the
exact minimax barrier \emph{in the GS graph}; this does not prove 19 optimal
among elementary AC paths, nor 66 moves a shortest certificate --- the
66-move certificate is \emph{derived from} an optimal-barrier GS path (its
expanded elementary peak, 24, is a separate quantity). All runs are
computer-assisted; source hashes, run records, and regression gates ship in
the archive.

\paragraph{Structural reading.}
The asymmetry has a local explanation visible in first neighborhoods: the
$\MS(3)$ forms (relator split $9+5$) have GS moves climbing only one unit
above their starting energy 14, while every first GS move from an $\MS(2)$
form (split $7+7$) jumps to energy 17 or 19. The exhaustion data also
separates the two searched $\MS(2)$ representatives: the sublevel region
reachable from $\PP{2}$ grows by $\times 5.97$ from energy $\le 25$ to
$\le 26$ ($424{,}352 \to 2{,}534{,}472$ settled states) versus
$\times 2.39$ for $\PP{1}$ ($356{,}836 \to 851{,}836$), and at the
budget-terminated barrier-27 processing the partial settled counts are
$5{,}640{,}276$ ($\PP{2}$) versus $1{,}987{,}856$ ($\PP{1}$) --- different
basin geometry, not merely a later copy of the $\MS(3)$ corridor. (All
figures re-derived from the committed run records by
\code{validate\_bottleneck.py}.)

\paragraph{Independent re-derivation of the bounded enumerations.}
At exclusive cap 25 the search settles exactly 135{,}892 states from the
$C_1$ representative and 168{,}304 from the $C_2$ representative ---
reproducing, with an independent search implementation (shared move
generator), the engine-relative component enumerations previously reported
for these classes, and adding the $\AK(3)$-side count of 1{,}532{,}732
states below total length 25. The three blocks are pairwise disjoint below
total length 25, 26, and (per the above) their pairwise joins require energy
$\ge 27$.

\section{The Two-Hump classification table and the $\sigma$-merge program}
\label{sec:twohump}

\cite{Fagan26} groups its 550 unsolved presentations into \emph{261
equivalence classes} (their Tables 2--3) without publishing the
classification algorithm or any connecting paths. We reverse-engineered what
the partition \emph{can} and \emph{cannot} be (script \code{reverse261.py};
hard assertion gates: 170 rows, 550 labeled cells, all 261 published
per-class occurrence counts reproduced exactly, $0/550$ length changes under
cyclic reduction).

\paragraph{Not a tested symmetry quotient.}
Canonical forms under (H1) rotation + inversion + swap give 550 distinct
keys; adding generator inversions (H2) or all signed permutations (H3) gives
275 keys that both split and merge their classes. None of the three
reproduces the partition; untested quotients cannot be excluded, but the
fingerprints point elsewhere: 354 of the 550 members carry class labels
shorter than their own (cyclically reduced) length, and their class
$14_2$ contains six members with three distinct H2-keys. The most
plausible explanation is that the table records connection information
computed in the course of their search. This constitutes strong evidence of
unpublished AC-connections --- including, plausibly, a path for our
Theorem~\ref{thm:sigma1} pair, since their table assigns both presentations
the same class label. Accordingly, the novelty we claim for Theorems
\ref{thm:sigma1}--\ref{thm:ms3b} is public replayability, not priority over
unpublished computations.

\paragraph{The $\sigma$-merge program.}
The 550 unsolved presentations pair perfectly under $\sigma$ into 275 pairs;
61 intra-class, \emph{214 crossing their class boundaries}. Joining their
partition with the $\sigma$-edges yields 133 components (validated in
\code{reverse261.py}).

\begin{corollary}[conditional]\label{cor:sigmamerge}
If $\sigma$ is AC-realized on all 214 cross-label pairs, the 261 classes of
\cite{Fagan26} collapse to at most 133.
\end{corollary}

Each certified cross-label $\sigma$-pair is an individually publishable
class merger. A pilot sweep (their-greedy engine, strict caps $L+11$): the
ten shortest cross-label pairs all \emph{exhaust} their bounded components
without contact --- engine-relative disconnections, consistent with
\S\ref{sec:minimax}'s finding that $\sigma$-realization on these classes
requires substantial energy excursions; the full stratified sweep is future
work. (Forward and reverse hunts are $\sigma$-mirror images --- proven by
the length-preserving automorphism argument and confirmed by identical state
counts --- halving the program's cost.)

\section{Negative results and computational practice notes}
\label{sec:negative}

\paragraph{Finite quotients can never obstruct these classes.}
By Borovik--Lubotzky--Myasnikov \cite{BLM11} the AC-components of normally
generating pairs in any finite group are determined by abelianization; the
exponent matrices of all presentations above are unimodular, so their images
lie in the generating pair's component in every finite quotient. Our
computational sieve over nine finite groups illustrates the theorem.
Obstructions, if any exist, must come from infinite quotients or genuinely
noncommutative invariants.

\paragraph{Unguided saturation fails far below the frontier.}
Prover9 \cite{McC}, on Lisitsa's implication-goal encoding of
AC-reachability (which has trivialized instances needing 1000+ moves under
expert guidance \cite{Lis18, Lis24, Lis25}), cannot re-prove a known
167-move trivialization in 2~h; the E prover \cite{Schulz19} (auto and
auto-schedule) fails in 15~min. Bounded, structured search
--- not saturation --- is currently the only working route at this scale.

\paragraph{A Prover9 pitfall worth recording.}
On Horn reachability encodings whose goal is a disjunction (e.g.\
``trivializable $\vee$ reaches $\AK(3)$''), clausification of the negated
goal yields one denial (negative clause) per disjunct; Prover9's default
\code{auto\_denials} setting then treats these as separate conclusions and
raises \code{max\_proofs} to their number (its log prints
\code{\%\ assign(max\_proofs, 4)}). This is documented behavior --- the
manual states that multiple denials in a Horn set are assumed to be
``separate conclusions'' with ``a separate proof of each''
\cite{McC} --- but for a disjunctive goal it silently changes the success
criterion from ``prove any'' to ``prove all.'' In our four-disjunct sanity
instance the unfixed run finds refutations of two disjuncts within
$0.01$~s (\code{\%\ Proof 1}/\code{2} blocks in the log) yet continues
searching for the remaining conclusions, exhausts its budget, and exits
SEARCH FAILED --- easy to misread as total failure, since the THEOREM
PROVED banner is printed only on successful termination. Under
\code{clear(auto\_denials).} \code{assign(max\_proofs,1).} the same input
terminates with THEOREM PROVED in $0.02$~s; single-disjunct goals are
unaffected. We initially misattributed these failures to encoding
fragility; the corrected calibration matrix and the demonstrating outputs
are in the repository.

\paragraph{Elementary-level cross-validation of \cite{Fagan26}.}
Our expansion pipeline regenerated their greedy solutions for two solved
instances (43 and 37 supermoves, the former matching their published table
exactly) and produced composed elementary ledgers passing our verifier (114
and 97 packaged moves). This validates their published solutions at the
elementary-move level, independently of their code.

\paragraph{Invalidated artifacts are retained, quarantined.}
An earlier attempt at an independent elementary-move class enumeration was
caught incomplete by its own validation gate; its artifacts are quarantined
in the archive (\code{invalidated/}) with a failure record. The
\S\ref{sec:minimax} re-derivation supersedes it for the component counts.

\section{Reproduction, audit scope, and data availability}\label{sec:repro}

One-command checking of the self-contained claims from the artifact archive
(\code{./repro.sh}, no failures masked; third-party-dependent analyses are
skipped, with a truthful PARTIAL banner, until their inputs are
regenerated): Theorems \ref{thm:sigma1}--\ref{thm:ms3b} replay; verifier
self-tests + accounting regression; Table-2 analysis with hard validation
gates; the Theorem~\ref{thm:minimax} record validator
(\code{validate\_bottleneck.py}: consistency of every \S\ref{sec:minimax}
figure against the committed run records, source hash included); the
committed AC-19 membership audit (\code{ac19\_audit.py}, hard assertions
under \cite{Fagan26}'s own canonicalizer: the positive anchor u066 present
in AC-19; all six benchmark presentations and $\AK(3)$ asserted absent from
both AC-19 (140{,}535 entries) and its extended release (156{,}762); dataset
hashes pinned). The archive contains: the verifier and self-test; all
certificates; both search engines (\code{ak\_bottleneck.py}, \code{bn.cpp})
with their regression gates;
\code{MANIFEST.json} (full SHA-256 hashes, commands, expected outputs) and
\code{LICENSE}; run records including the frozen $\sigma$-pilot and the
\S\ref{sec:minimax} exhaustion records; provenance notes (with SHA-256
hashes) for the third-party inputs used by the parsers; the quarantined
invalidated artifacts with their failure records; and the claims ledger
(\code{RESULTS.md}) documenting every retracted claim and its correction.
Repository: \url{https://github.com/joe-carr-data/ac-certificates} (release tag \code{v1.0}); permanent archive: \url{https://doi.org/10.5281/zenodo.21499081} (DOI \code{10.5281/zenodo.21499081}).

\paragraph{Audit scope for the novelty statement (as of 2026-07-22).}
arXiv (2408.15332 v1/v2; 2606.21611 incl.\ appendices; 1609.00325), the
AC-Solver repository (HEAD and history), its W\&B artifacts, the ACSolverX
repository and both AC-19 releases (membership audited under their
canonicalizer), the MATH-AI @ NeurIPS 2025 workshop paper on this problem
family, recorded talks, and a commissioned deep-research literature sweep.
``First public'' means: not found in any of these; an unindexed or private
artifact cannot be excluded, and \cite{Fagan26}'s table is itself evidence
that stronger unpublished computations exist (\S\ref{sec:twohump}).

\paragraph{AI assistance disclosure.}
Search code, certificates, verification tooling, and this manuscript were
produced in a human-directed collaboration with AI systems; every
mathematical claim rests on the replayable certificates and scripts in the
archive, not on AI assertion. The development record, including all
corrected intermediate claims, is documented in the archived claims ledger
(\code{RESULTS.md}).

\end{document}